\documentclass[12pt]{amsart}
\usepackage{amsmath,amssymb,amsfonts,latexsym}
\usepackage{mathrsfs}
\usepackage[dvips]{geometry}
\usepackage{array}
\usepackage{epsf}
\usepackage{epsfig}
\newtheorem*{lemma}{Lemma}
\newtheorem*{prop}{Proposition}
\newtheorem*{thm}{Theorem}
\newtheorem*{cor}{Corollary}

\newcommand{\nc}{\newcommand}
\nc{\ad}{\operatorname{ad}} \nc{\tr}{\operatorname{tr}}
\nc{\tp}{\operatorname{top}} \nc{\rank}{\operatorname{rank}}
\nc{\Inj}{\operatorname{Inj}} \nc{\Hom}{\operatorname{Hom}}
\nc{\End}{\operatorname{End}} \nc{\supp}{\operatorname{supp}}
\nc{\ch}{\operatorname{ch}} \nc{\im}{\operatorname{im}}
\nc{\gr}{\operatorname{gr}}\nc{\Id}{\operatorname{Id}}

\begin{document}

\title[Zhelobenko Invariants]{Analogue Zhelobenko Invariants and the Kostant Clifford Algebra Conjecture}
\author[Anthony Joseph]{Anthony Joseph}
%
\footnotetext[1]{Work supported in part by
Israel Science Foundation Grant, no. 710724.}
\date{\today}
\maketitle

\vspace{-.9cm}\begin{center}
Donald Frey Professional Chair\\
Department of Mathematics\\
The Weizmann Institute of Science\\
Rehovot, 76100, Israel\\
anthony.joseph@weizmann.ac.il
\end{center}\



 \

\textbf{Abstract}  Let $\mathfrak g$ be a complex simple Lie algebra and $\mathfrak h$ a Cartan subalgebra. The Clifford algebra $C(\mathfrak g)$ of $\mathfrak g$ admits a Harish-Chandra map.  Kostant conjectured (as communicated to Bazlov in about 1997) that the value of this map on a (suitably chosen) fundamental invariant of degree $2m+1$ is just the zero weight vector of the simple $(2m+1)$-dimensional module of the principal s-triple obtained from the Langlands dual $\mathfrak g^\vee$.  Bazlov \cite {B1} settled this conjecture positively in type $A$.

The Kostant conjecture was reformulated (Alekseev-Bazlov-Rohr \cite {AM,B2,R}) in terms of the Harish-Chandra map for the enveloping algebra $U(\mathfrak g)$ composed with evaluation at the half sum $\rho$ of the positive roots.

 In an earlier work we settled \cite {J2} an analogue of the Kostant conjecture obtained by replacing the Harish-Chandra map by a ``generalized Harish-Chandra" map whose image is described via Zhelobenko invariants.   Here we show that there are analogue Zhelobenko invariants which describe the image of the Harish-Chandra map.  Following this a similar proof goes through.

\section{Introduction}

\

The base field $k$ is assumed algebraically closed of characteristic zero throughout.

\subsection{}\label{1.1}

Let $\mathfrak g$ be a simple Lie algebra with $U(\mathfrak g)$ its enveloping algebra.  Let $V$ be a simple finite dimensional $\mathfrak g$ module.

This work is a sequel to \cite {J2} in which we settled an ``analogue" Kostant Clifford algebra conjecture phrased in terms of the generalized Harish-Chandra map $\hat{\Phi}$ acting on $(V\otimes U(\mathfrak g))^\mathfrak g$ .  In this the Zhelobenko invariants describe the image of $\hat{\Phi}$.  Our method of proof used a partial description of these invariants together with a symmetric algebra computation occurring in the works of Bazlov \cite {B2} and Alekseev's student Rohr \cite {R}.

The Kostant Clifford algebra conjecture as rephrased by Alekseev and Bazlov \cite {B2} involves the usual Harish-Chandra map $\Phi$ also acting on $(V\otimes U(\mathfrak g))^\mathfrak g$. We note that there exists ``analogue" Zhelobenko invariants which describe the image of $\Phi$.  This came rather as a surprise though it might have been anticipated by the close relation between the corresponding determinants noted in \cite {J1}.  This results from fairly easy $\mathfrak {sl}(2)$ computations.  From it we deduce an analogue of \cite [Cor. 2.5]{J2} in which the denominator in the left hand side of \cite [Eq. (9)]{J2} is changed through the replacement of $1$ by $-1$.  (Here one may recall that the symmetric algebra analogue \cite [Eq. (12)]{J2} of this result is obtained through the replacement of $1$ by $0$.)

 Using the above result the proof of the Kostant conjecture follows exactly the proof given in \cite {J2} for the analogue Kostant conjecture. 

\subsection{}\label{1.2}

Let $\mathfrak h$ Cartan subalgebra of $\mathfrak g$.

Let $\mathscr F$ be the canonical (resp. degree) filtration on $U(\mathfrak g)$.

Let $\Delta \subset \mathfrak h^*$ be the set of non-zero roots of $\mathfrak g$, $\Delta^+$ a choice of positive roots and $\pi=\{\alpha\}_{i\in I}:I:=1,2,\ldots,\ell$, the corresponding set of simple roots and $\varpi_i:i \in I$ the corresponding set of fundamental weights. Let $x_\alpha:\alpha \in \Delta$ form (part) of a Chevalley basis for $\mathfrak g$. Set $x_i=x_{\alpha_i},y_i=x_{-\alpha_i}, h_i=[x_i,y_i]$.  They form the basis of an $\mathfrak {sl}(2)$ subalgebra $\mathfrak s_i$ of $\mathfrak g$.  The simple coroots $h_i=\alpha_i^\vee:i \in I$ complete the Chevalley basis.

Given $\gamma \in \Delta$, let $\gamma^\vee$ denote the corresponding coroot (for example identified with $2\gamma/(\gamma,\gamma)$ through the Cartan inner product on $\mathfrak h^*$).  Let $\Delta^\vee$ (resp. $\Delta^{\vee+}$) denote the set of coroots (resp. positive coroots).

Let $s_i$ be the simple reflection defined by $\alpha_i\in \pi$ and $W$ the group they generate.  Let $\rho$ denote the half sum of the elements of $\Delta^+$.

%

Let $\mathfrak g^\vee$ denote the Langlands dual of $\mathfrak g$. Its roots are the coroots of $\mathfrak g$. One may identify a Cartan subalgebra $\mathfrak h^\vee$ of $\mathfrak g^\vee$ with $\mathfrak h^*$ and then with $\mathfrak h$ though the Cartan scalar product.  Let $e^\vee,h^\vee,f^\vee$ be a principal s-triple for $\mathfrak g^\vee$ chosen so that $h^\vee \in \mathfrak h^\vee$. One may identify $h^\vee$ with $\rho$.

Through the above identifications we obtain a filtration on $\mathfrak h$ through the adjoint action of $e^\vee$, namely $\mathscr F^m(\mathfrak h)= \{h \in \mathfrak h|(e^\vee)^{m+1}h=0\}$.

Recall that there is a translated Weyl group action on $\mathfrak h^*$ defined by $w.\lambda=w(\lambda +\rho)-\rho$.  This induces a translated Weyl group action on $S(\mathfrak h)$ by identifying the latter with the algebra of polynomial functions on $\mathfrak h^*$ and transport of structure.

\subsection{}\label{1.3}

Let $V$ be a finite dimensional simple $\mathfrak g$ module and $V_\mu$ its weight subspace of weight $\mu \in \mathbb Z \pi$.  Consider $U(\mathfrak g)$ as a $\mathfrak g$ module under adjoint action and then the tensor product $V\otimes U(\mathfrak g)$ as a $U(\mathfrak g)$ module through the coproduct.   Let $\varphi:U(\mathfrak g) \rightarrow U(\mathfrak h)$ be the Harish-Chandra map and $\varphi_\lambda$ its composition with evaluation at $\lambda \in \mathfrak h^*$.  Set $\Phi^V=Id\otimes \varphi:V\otimes U(\mathfrak g)\rightarrow V_0 \otimes U(\mathfrak h)$ and $\Phi^V_\lambda=Id\otimes \varphi_\lambda:V\otimes U(\mathfrak g)\rightarrow V_0$.  We omit the $V$ superscript exactly when $V$ is the adjoint module which we identify with $\mathfrak g$ and then $V_0$ with $\mathfrak h$.

\subsection{}\label{1.4}

The Kostant conjecture (as rephrased by Alekseev and Bazlov \cite {B2}) asserts that $\Phi_{s\rho}(\mathfrak g \otimes \mathscr F^m U(\mathfrak g))^\mathfrak g =  \mathscr F^m(\mathfrak h), \forall m \in \mathbb N$, for all $s$ in a cofinite subset of $k$.  (For $\mathfrak g= \mathfrak {sl}(n)$, the case $s=1$ is resolved in Bazlov's thesis \cite {B1} and for more general $s$ it is resolved in recent work of Alekseev and Moreau \cite {AM}.)  At least when $\mathfrak g$ is simply-laced one may easily formulate a more general conjecture in which the adjoint module is replaced by an arbitrary simple finite dimensional $\mathfrak g$ module $V$.  However as pointed out to me by Alekseev, this fails badly even though its symmetric algebra analogue still holds.  It must therefore be presumed that the Kostant conjecture is a hard problem in which some special feature of the adjoint module must be used.  However to begin with we shall place ourselves in this more general context.

\subsection{}\label{1.5}

Zhelobenko \cite {Z} defined a set of operators on $V\otimes U(\mathfrak g)$ acting trivially on invariants with the marvellous property that they factor through the generalized Harish-Chandra map $\hat{\Phi}$.  Khoroshkin, Nazarov and Vinberg \cite {KNV} proved the remarkable fact that $\hat{\Phi}(V\otimes U(\mathfrak g))^\mathfrak g$ is exactly the set of Zhelobenko invariants in $V_0\otimes U(\mathfrak h)$.  We call this the KNV theorem.  This fact was used in \cite {J2} to settle the analogue Kostant conjecture and here we took advantage of considerable simplification that occurs when $V$ is the adjoint module.

Although at first sight the above method would not seem to work for the Kostant conjecture,  it turns out that we may (rather easily) define a set $\Xi = \{\xi_i:i \in I\}$ of analogue Zhelobenko operators \textit{directly} on $V_0\otimes U(\mathfrak h)$ with the property that $\Phi(V\otimes U(\mathfrak g))^\mathfrak g= (V_0\otimes U(\mathfrak h))^\Xi$.   Then following \cite {J2} we describe rather explicitly the right hand side when $V$ is the adjoint module obtaining the analogue of \cite [Cor. 2.5]{J2} mentioned in \ref {1.1}.  From this the proof of the Kostant conjecture results by following precisely the analysis in \cite {J2}.

\subsection{}\label{1.6}

The four sections of this manuscript which describe our proof of the Kostant conjecture (together with a fifth section concerning a question of Hitchin which
has since been changed) were sent to Alekseev and Moreau with whom I even proposed joint publication, as well as to Kostant and Kumar.  Alekseev and Moreau
sent me in return a preliminary version of [AM2] announcing a proof of
the Kostant conjecture.  Their proof depends crucially on my earlier
paper \cite {J2} and proceeds by extracting a very special case of the
comparison between the Harish-Chandra and generalized Harish-Chandra maps
established in Section 3 below.

\

\textbf{Acknowledgements}. The origins of this work go back to discussions with Alekseev and Nazarov.  I would also like to thank Shrawan Kumar for bringing the Hitchin conjecture \cite [2.2, Remark]{H} to my attention though it transpired that this was a slightly different question.

\section{The Analogue Zhelobenko Operators}

In this and later sections we shall use the \textit{same} symbols to denote the \textit{analogue} Zhelobenko operators (resp. analogue Zhelobenko invariants) as was used in \cite {J2} to denote the Zhelobenko operators (resp. Zhelobenko invariants).   This will avoid wasting extra symbols and as this is a separate paper no confusion should result.

\subsection{}\label{2.1}

Define $V$ as in \ref {1.4}.  For all $i \in I$ set $V_{\mathbb Z\alpha_i}=\oplus_{n \in \mathbb Z}V_{n\alpha_i}$.

Zhelobenko operators are really about $\mathfrak {sl}(2)$ computations.  Thus let $\mathfrak r_i$ (resp. $\mathfrak m_i, \mathfrak m^-_i$) denote the Levi factor
(resp. nilradical, opposed nilradical) of the minimal parabolic defined by $i\in I$.  One has  $\mathfrak r_i = \mathfrak h + \mathfrak s_i$. The canonical projection $\varphi^i: U(\mathfrak g)\rightarrow U(\mathfrak g)/\mathfrak m^-_iU(\mathfrak g)+U(\mathfrak g)\mathfrak m_i= U(\mathfrak r_i)$ is $\mathfrak r_i$ equivariant and so $\Phi^i:=Id \otimes \varphi^i$ maps $(V\otimes U(\mathfrak g))^{\mathfrak r_i}$ into $(V_{\mathbb Z\alpha_i}\otimes U(\mathfrak r_i))^{\mathfrak r_i}$.

Now let $\varphi_i:U(\mathfrak r_i)\rightarrow U(\mathfrak h)$ be the Harish-Chandra map and set $\Phi_i=Id \otimes \varphi_i: V_{\mathbb Z \alpha_i}\otimes U(\mathfrak r_i)\rightarrow V_0 \otimes U(\mathfrak h)$.
Obviously $\varphi =\varphi_i \varphi^i, \Phi= \Phi_i \Phi^i$.

\subsection{}\label{2.2}

The special case of the \cite {PRV} applied to the case $\mathfrak g =\mathfrak {sl}(2)$ computes $\Phi_i(V_{\mathbb Z\alpha_i}\otimes U(\mathfrak r_i))^{\mathfrak r_i}$.  Indeed let $V(m)$ denote a simple $\mathfrak r_i$ of $V_{\mathbb Z\alpha_i}$ of dimension $2m+1$.  Set $\psi_{n,i}=\prod_{m=1}^n(h_i-(m-1))$.  Then from \cite {PRV} (for which a simple proof is given in \cite {J0}) or directly one obtains
$$\Phi_i(V(m)\otimes U(\mathfrak r_i))^{\mathfrak r_i}=U(\mathfrak h)^{s_i.}\psi_{n,i}. \eqno {(1)}$$

\subsection{}\label{2.3}

The above result may be cast into the following form. Let $v_0(m)$ be a non-zero element of $V_0(m)$.  Let $K$ be the fraction field of the commutative domain $U(\mathfrak h)$.  Define an endomorphism $\xi_i$ of $V(m)_0\otimes K$ by $$\xi_i(v_0(m)\otimes q)= v_0(m) \otimes \frac {\psi_{n,i}}{s_i.\psi_{n,i}}s_i.q$$.

It is clear that $v_0(m)\otimes q$ is $\xi_i$ invariant if and only if $$q=\frac {\psi_{n,i}}{s_i.\psi_{n,i}}s_i.q. \eqno {(2)}$$.

Here one may remark that $s_i.\psi_{n,i}=(-1)^n\prod_{m=1}^n(h_i+(m+1))$, whose zeros are distinct from those of $\psi_{n,i}$.  Thus $(2)$ holds if and only if $\psi_{n,i}$ divides $q$ and $p:=q/\psi_{n,i}$ satisfies $s_i.p=p$.

Now extend $\xi_i$ to an endomorphism of $V_0\otimes K$ by linearity using the direct decomposition of $V_{\mathbb Z \alpha_i}$ into simple $\mathfrak r_i$ modules and linearity. Then by the above we have the following

\begin {lemma}  $\Phi_i(V_{\mathbb Z\alpha_i}\otimes U(\mathfrak r_i))^{\mathfrak r_i}= (V_0\otimes U(\mathfrak h))^{\xi_i}$.
\end {lemma}

\subsection{}\label{2.4}

Set $\Xi= \{\xi_i:i \in I\}$. We obtain the following analogue of the KNV theorem.

\begin {thm}  $\Phi^V$ is an isomorphism of $(V\otimes U(\mathfrak g))^\mathfrak g$ onto $(V_0\otimes U(\mathfrak h))^\Xi$.
\end {thm}

\begin {proof}  By Lemma \ref {2.3} we have $\Phi((V\otimes U(\mathfrak g))^\mathfrak g) \subset (V_0\otimes U(\mathfrak h))^\Xi$.  Injectivity which is elementary follows exactly as in the proof of \cite [Lemma 2.6]{J1}.  Surjectivity follows exactly as in \cite [Sect. 3]{J1} using the PRV determinant instead of $D$ (see \cite [3.3,3.6]{J1}.  Surjectivity also follows by the method of \cite {KNV} using passage to the graded objects and calculating the image of the corresponding Chevalley restriction map (which is \textit{not} surjective).
\end {proof}

\section{The Analogue of the Basic Identity}

We now establish the analogue of the basic identity \cite [Cor. 2.5]{J2}.  Via the results of Section 2 the proof is much the same.  The term analogue will be dropped and we shall use the same notation for these analoguous objects. The formulae are similar but slightly different.

\subsection{}\label{3.1}

It follows from the construction of the Zhelobenko operators that we have the following version of \cite [Lemma 2.3]{Z} (see also \cite [Lemma 2.2]{J2})

\begin {lemma}  For all $i \in I$, $a \in V_0 \otimes S(\mathfrak h)$, $b \in S(\mathfrak h)$ one has $$\xi_i(ba)=(s_i.b)\xi_i(a). \eqno {(3)}$$.
\end {lemma}

\subsection{}\label{3.2}

It follows from the Theorem \ref {2.4} and the previous lemma that $(V_0 \otimes S(\mathfrak h))^\Xi$ is a free $S(\mathfrak h)^{W.}$ module on $\dim V_0$ generators. Their leading order terms are just the set of generators for the free $S(\mathfrak h)^W$ module obtained as the image of the Chevalley restriction map of $(V\otimes S(\mathfrak g))^\mathfrak g$ into $(V_0\otimes S(\mathfrak h))^W$ (which is \textit{not} generally surjective). In the case when $V$ is the adjoint module we may label these generators by $I$, specifically as $\{J_i\}_{i\in I}$.

\subsection{}\label{3.3}

From now we just take $V$ to be the adjoint module.  In this case the Chevalley restriction map \textit{is} surjective.  Let $J$ be a Zhelobenko invariant.  We may write
$$J=\sum_{i\in I} \varpi_i \otimes q_i,$$
for some $q_i \in S(\mathfrak h)$.

We now prove the following analogue of \cite [Proposition 2.4]{J2}.

\begin {prop}

\

(i) $s_i.q_i=-\frac {h_i+2}{h_i} q_i$.

\

(ii)  $q_i$ is divisible by $h_i$ and $p_i:=q_i/h_i$ is $s_i.$ invariant.

 \

(iii)  $q_j-s_i.q_j=\frac {1}{2} h_j(\alpha_i)(q_i-s_i.q_i), \forall i,j \in I\in S(\mathfrak h)$.

\end {prop}

\begin {proof} One has $$\alpha_i=2\varpi_i-\alpha_i^\perp,$$
with $\alpha_i^\perp :=-\sum_{j\in I \setminus \{i\}}h_j(\alpha_i)\varpi_j$, being orthogonal to $\alpha_i$.

In view of the definition of $\xi_i$, we obtain
$$2\xi_i(\varpi_i\otimes 1)=\eta_i(\alpha_i^\perp-\alpha_i)\otimes 1=(\alpha_i^\perp-\alpha_i)\otimes 1 -2\alpha_i\otimes h_i^{-1}=(\alpha_i^\perp\otimes 1)-\alpha_i\otimes \frac{h_i}{h_i+2}.$$

In view of $(3)$ this gives
$$\xi_i(\varpi_i\otimes q_i)=\frac{1}{2}(\alpha^\perp_i\otimes s_i.q_i)-(\varpi_i-\frac {\alpha_i^\perp}{2})\otimes \frac {h_i}{h_i+2}s_i.q_i.\eqno {(4)}$$

On the other hand by definition of $\xi_i$ and $(3)$ again we have
$$\sum_{j\in I\setminus \{i\}}\xi_i(\varpi_j\otimes q_j)= \sum_{j\in I\setminus \{i\}}\varpi_j\otimes s_i.q_j. \eqno {(5)}$$

In the sum of the right hand sides of $(4)$ and $(5)$ the coefficient of $\varpi_i$ is just $-\frac {h_i}{h_i+2}s_i.q_i$ and so the invariance of the sum of the left hand sides, which is $J$, implies that $q_i=-\frac {h_i}{h_i+2}s_i.q_i$. Consequently $q_i$ is divisible by $h_i$. This gives (i). (ii) follows from (i).

Again by the invariance of $J$ and equating the coefficients of $\varpi_j$ we obtain (iii) from $(4)$ and $(5)$.

\end {proof}

\textbf{Remark}.   Although (i) differs slightly from the first part of \cite [Proposition 2.4]{J2}, (iii) is exactly the same as \cite [Eq. 6]{J2}.

\subsection{}\label{3.4}

By \cite [Eq. (2)]{J2} or directly we obtain $s_i.h_j=h_j-h_j(\alpha_i)(h_i+1)$, for all $i,j \in I$.  Recalling that $s_i.$ acts by automorphisms, Proposition \ref {3.3}(iii) gives
$$h_j(p_j-s_i.p_j)= h_j(\alpha_i)(h_i+1)(p_i-s_i.p_j). \eqno {(6)}$$

\subsection{}\label{3.5}

 As in \cite [2.5]{J1} we use the automorphism $\theta$ of $S(\mathfrak h)$ defined by $\theta(q)(\lambda)=q(\lambda + \rho)$, which has the property that $w.\theta(q)=\theta(wq)$.  Observe further that $\theta(h_i+m)=h_i+m+1$.  Now define new polynomials $P_i:=\theta^{-1}(p_i)$.  Substitution in $(6)$ gives
$$(h_j-1)(P_j-s_iP_j)=h_j(\alpha_i)h_i(P_i-s_iP_j). \eqno {(7)}$$

Now following \cite {BGG} we introduce linear operators $A_i:i\in I$ on $S(\mathfrak h)$ by the formulae
$$A_if:=\frac{f-s_if}{h_i}, \forall f \in S(\mathfrak h). \eqno {(8)}$$

This gives the

\begin {cor}  For all $i,j \in I$ one has
$$A_iP_j=h_j(\alpha_i)\frac{P_i-P_j}{-1+s_i(h_j)}. \eqno {(9)}$$
\end {cor}

\begin {proof}  Subtract $h_j(\alpha_i)h_i(P_j-s_iP_j)$ from both sides of $(7)$. Since $h_j-1-h_j(\alpha_i)h_i=-1+s_i(h_j)$, the assertion results.
\end {proof}

\subsection{}\label{3.6}

Equation $(9)$ is what we must solve in order to determine the (analogue) Zhelobenko invariant $J$. Of course this is not too easy as there are infinitely many solutions.  The simplest solution is when all the $P_j$ equal $1$.

Equation $(9)$ is remarkably similar to \cite [Eq. (9)]{J2} which we must solve to determine the Zhelobenko invariants.  The leading order terms $P_i^0$ in both cases both satisfy the same equation (\cite [Eq. (12)]{J2}) as of course we expect.

\section{Proof of the Kostant Conjecture}

\subsection{}\label{4.1}

The result stated in $(9)$ and its close relation to \cite [Eq. (9)]{J2} means that the proof of the Kostant conjecture now follows exactly the same reasoning as that of the analogue Kostant conjecture.  Indeed as already pointed out in \cite [8.1] {J2} we may replace $1$ occurring in the denominator of \cite [Eq. (9)]{J2} by any scalar $c$ and still obtain the conclusion of \cite [Proposition 7.8]{J2} but with $1$ occurring in the denominator by $c$.  The case $c=0$ gave \cite [Proposition 8.1]{J2} which corresponds to the symmetric algebra case.  Again the fact that this scalar is $c$ rather than $1$ does not effect the conclusion of \cite [Proposition 8.2]{J2} from which \cite [Proposition 8.6]{J2} results.  This immediately gives the Kostant conjecture as formulated in \ref {1.4}.  Here the only difference is the cofinite subset of $k$ for which the two filtrations are the same.  As already noted in \cite [8.7, Remark 1]{J2} for the analogue Kostant problem this set is determined by the zeros of $D$ (as defined in \cite [3.3]{J1}) whereas for the Kostant problem it is determined by the zeros of the KPV determinant, as described in \cite [3.3]{J1}.  In both cases the cofinite set contains all the positive integers.  Indeed the only ``bad" integers are $-2$ and $0$ respectively.

\section{Added Remarks}
\subsection{}\label{5.1}

For all $m \in \mathbb N$, set $V(m)^\vee$ denote the direct sum with respect to the action of the principal s-triple for $\mathfrak g^\vee$ into a direct sum of simple $\mathfrak {sl}(2)$ modules of dimension $2m+1$. These modules, being pairwise non-isomorphic, are pairwise orthogonal with respect to the Killing form for $\mathfrak g^\vee$. Identify $\mathfrak h^\vee$ with $\mathfrak h^*$ and hence with $\mathfrak h$ through the Killing form.  It follows from the above that the zero weight subspaces $V(m)^\vee_0: m \in \mathbb N$ of $V(m)^\vee$ form an orthogonal direct sum decomposition of $\mathfrak h$.
\subsection{}\label{5.2}

 Kostant \cite [Thm. 35]{K2} showed, following an analogous result (Hopf-Koszul-Samelson, see \cite [Thm. 23]{K2} for $\bigwedge^* \mathfrak g$, that $C(\mathfrak g)^\mathfrak g$ is again a Clifford algebra over a subspace $P$ of so-called primitive elements.

The triangular decomposition of $\mathfrak g$ gives a triangular decomposition of $C(\mathfrak g)$ and hence a Harish-Chandra projection $\phi$ of $C(\mathfrak g)$ onto $C(\mathfrak h)$.  Bazlov \cite [Prop. 4.5] {B2} first wrote out a proof that $\phi$ restricts to an isomorphism of $P$ onto $\mathfrak h$.  It seems that the result was known to Kostant and in any case obtains rather easily from the analysis in \cite [Sect. 6]{K2}.

Let $\mathscr F$ denote the filtration on $\mathfrak h$ given by $\mathscr F^m(\mathfrak h):= \Phi_\rho(\mathfrak g \otimes \mathscr F^mU(\mathfrak g))^\mathfrak g: m \in \mathbb N$. 

The truth of the Kostant conjecture (as established in \ref {4.1}) means that $\mathscr F^m(\mathfrak h)\subset \oplus_{n\leq m}V(n)^\vee_0$.

\subsection{}\label{5.3}

Let $m_i:i \in I$ denote the exponents of $\mathfrak g$ (which are the same for $\mathfrak g^\vee$) taken in order of increasing value.

Choose a basis for $P$ by successively taking elements $c_{2m_i+1}:i \in I$ of degree $2m_i+1$ such that their images $h_{m_i}:=\phi(c_{2m_i+1})$ are pairwise orthogonal to the $m_j:j<i$. 

The connection between the original Kostant conjecture and its presentation in \ref {4.1} was explained in \cite {B2}.   They may be succinctly summarized as follows. Through the construction of $P$ given in \cite [Sect. 6]{K2}, the map $\delta: U(\mathfrak g)\rightarrow C(\mathfrak g)$ defined in \cite [Eq. (70)]{K2} and notably \cite [Thm. 74]{K2} describing the image of $\delta$, it is immediate from \ref {5.1}, \ref {5.2} that the $h_{m_i}:m_i=m$ span $V(m)_0^\vee$.  This result is the original form of the Kostant conjecture.  However I have been unable to ascertain whether Kostant intended a more precise result taking the $c_{2m_i+1}:i \in I$  to be those described in \cite [Eq. (326)]{K2}, where the analysis suggests that he was aware of the symmetric algebra analogue \cite [Eq (11)]{J2} of his conjecture.

\subsection{}\label{5.4}

It should be noted that although Bazlov's manuscript \cite {B2} contains many useful facts, his handling of lower order terms, necessarily introduced by Clifford/Enveloping algebra considerations, is erroneous.  Our analysis of these terms in \cite {J2} and in the present manuscript requires an \textit{extensive} knowledge of the images of the Harish Chandra (or generalized Harish-Chandra maps) and this constitutes what we consider to be the major breakthrough on the Kostant conjecture enabling one to pass (both here and in \cite {AM}) from the symmetric algebra to the Clifford/Enveloping algebra case.  These lower order terms are obtained from the $P_i:i \in I$ and do not satisfy the key assertions in the proof \cite [Thm. 5.5]{B2}.

\end{document}